\documentclass[11pt,reqno,]{amsart}
\usepackage[
    top=1.5in, bottom=1.5in, inner=1.4in, outer=1.4in,
    marginparwidth=2cm %necessary to avoid breaking fixme notes.
    ]{geometry}

\usepackage{amssymb,latexsym, amsmath, amsxtra, mathrsfs, bm}
\usepackage[dvips]{graphics}
\usepackage[T1]{fontenc}

\usepackage[all]{xy}
\usepackage{xcolor}
\usepackage{subcaption}
\usepackage{verbatim}
\usepackage[abs]{overpic}

\usepackage{hyperref}
\hypersetup{colorlinks, linkcolor=red, citecolor=green}

\usepackage{mathtools}
\usepackage{enumitem}
\usepackage{bbm} 

\DeclareMathAlphabet{\mathbbold}{U}{bbold}{m}{n}

\allowdisplaybreaks[3]

\makeatletter
\def\th@plain{%
	\thm@notefont{}% same as heading font
	\itshape % body font
}
\def\th@definition{%
	\thm@notefont{}% same as heading font
	\normalfont % body font
}
\makeatother
\theoremstyle{plain}
        \newtheorem{theorem}{Theorem}[section]
        \newtheorem*{theorem*}{Theorem}
        \newtheorem*{theoremM}{Main Theorem}

        \newtheorem{lemma}[theorem]{Lemma}
        
        \newtheorem{cor}[theorem]{Corollary}

\theoremstyle{definition}

\theoremstyle{remark}

\newtheoremstyle{maintheorem}%
  {6pt}
  {6pt}% 
  {\itshape}
  {}
  {\bfseries}
  {.}
  { }
  {\thmname{#1}\thmnumber{ #2}\thmnote{ {\notefont(#3)}}}

\numberwithin{equation}{section}
\numberwithin{theorem}{section}
\numberwithin{table}{section}

\makeatletter
\def\th@plain{%
	\thm@notefont{}% same as heading font
	\itshape % body font
}
\def\th@definition{%
	\thm@notefont{}% same as heading font
	\normalfont % body font
}
\makeatother
\renewcommand{\le}{\leqslant}
\renewcommand{\leq}{\leqslant}
\renewcommand{\ge}{\geqslant}
\renewcommand{\geq}{\geqslant}

\newcommand{\Var}{\operatorname{Var}}

\newcommand{\id} {\operatorname{id}}

\newcommand{\Fix}{\mathrm{Fix}}

\newcommand{\eps}{\varepsilon}
\newcommand{\R}{\mathbb{R}}
    
\newcommand{\C}{\mathbb{C}}      
\newcommand{\Pillow}{\mathbb{P}}

\newcommand{\N}{\mathbb{N}}      
\newcommand{\Z}{\mathbb{Z}}      
\newcommand{\T}{\mathbb{T}}      % Torus

\renewcommand{\hat}{\widehat}

\providecommand{\abs}[1]{\lvert#1\rvert}
\providecommand{\Absbig}[1]{\bigl\lvert#1\bigr\rvert}
\providecommand{\Absbigg}[1]{\biggl\lvert#1\biggr\rvert}
\providecommand{\AbsBig}[1]{\Bigl\lvert#1\Bigr\rvert}

\providecommand{\norm}[1]{\|#1\|}
\providecommand{\Normbig}[1]{\bigl\|#1\bigr\|}

\newcommand{\CC}{\mathcal{C}}

\newcommand{\bigO}{\mathcal{O}}

\renewcommand{\=}{\coloneqq}

\newcommand{\cC}{\mathcal{C}}

\newcommand{\cE}{\mathcal{E}}

\newcommand{\cL}{\mathcal{L}}

\newcommand{\cO}{\mathcal{O}}

\newcommand{\bP}{\mathbb P}

\newcommand{\bbp}{\mathbbm{p}}
\newcommand{\vf}{\varphi}
\newcommand\Id{{\mathbbm{1}}}
\newcommand{\bE}{\mathbb{E}}

\title[Normal distribution of periodic Lyapunov exponents]{Normal distribution of Lyapunov exponents of periodic orbits for expanding circle maps}
\begin{document}

	\author[K.~Drach, Z.~Fu, V.~Kaloshin, Z.~Li, C.~Liverani]{Kostiantyn~Drach \and Zhi~Fu \and Vadim~Kaloshin \and Zhiqiang~Li \and Carlangelo~Liverani}
    
%\begin{comment}

	\address{Kostiantyn~Drach, Universitat de Barcelona, Gran Via de les Corts Catalanes, 585, 08007 Barcelona, Spain \newline 
    \indent Centre de Recerca Matemàtica, Edifici C, Carrer de l'Albareda, 08193 Bellaterra, Spain}
	\email{kostiantyn.drach@ub.edu}
    
	\address{Zhi~Fu, School of Mathematical Sciences, Peking University, Beijing 100871, China}
	\email{zhifu@stu.pku.edu.cn}

	\address{Vadim~Kaloshin, Institute of Science and Technology Austria, Am Campus 1, 3400 Klosterneuburg, Austria}
	\email{vadim.kaloshin@gmail.com}
    
	\address{Zhiqiang~Li, School of Mathematical Sciences \& Beijing International Center for Mathematical Research, Peking University, Beijing 100871, China}
	\email{zli@math.pku.edu.cn}

    \address{Carlangelo~Liverani, Dipartimento di Matematica, II Universit\`a di Roma (Tor Vergata), Via della Ricerca Scientifica, 00133 Roma, Italy}
	\email{liverani@mat.uniroma2.it}

%\end{comment}
 %   \address{Department of Mathematics, William E. Kirwan Hall, 4176 Campus Dr, University of Maryland, College Park, MD 20742, USA}
  %  \email{cliveran@umd.edu}

	

	\thanks{ZF and ZL are partially supported by the Beijing Natural Science Foundation (JQ25001) and the National Natural Science Foundation of China (12471083, 12090010, 12090015). KD is partially supported by the Agencia Estatal de Investigación (PID2023-147252NB-I00) and by the Severo Ochoa–María de Maeztu Program for Centers and Units of Excellence (CEX2020-001084-M). VK and KD acknowledge the support from the ERC Advanced Grant SPERIG (\#885707). CL was supported by the PRIN Grants 2017S35EHN and 2022NTKXCX, and by the MIUR Excellence Department Project Math@TOV awarded to the Department of Mathematics, University of Rome Tor Vergata. CL also acknowledges membership in the GNFM/INDAM.}
    
\begin{abstract}
For a smooth expanding circle map, we show that the empirical distribution of Lyapunov exponents of periodic points of any fixed period is close to normal, with an error that decreases as the period grows. This establishes a version of the Central Limit Theorem for such finite periodic orbits.
\end{abstract}

\maketitle

\setcounter{tocdepth}{2}

\enlargethispage{2\baselineskip}
    \vspace*{1\baselineskip}
    \thispagestyle{empty}

%	\tableofcontents

\section{Main result}
\label{sec_expanding_maps}
In this article, for a smooth orientation-preserving expanding circle map $f \colon \T \to \T$, we show that the Lyapunov exponents of its periodic orbits of any \emph{fixed} period distribute normally, with a small error depending polynomially on the period. This result is supported by several computer experiments illustrated in Figure~\ref{Fig:Exp}.

Let $\Fix(f^n)$ be the set of periodic points of period $n$ (we allow for it to be not the smallest period). If $K \ge 2$ is the \emph{degree} of $f$, then $\Fix(f^n)$ contains exactly $K^n-1$ elements. For $p \in \Fix(f^n)$, define
\begin{equation*}
\chi_n(p) 
\= \frac{1}{n} \ln (f^n)'(p) 
= \frac{1}{n} \sum_{k=0}^{n-1} \bigl( \ln f' \circ f^k \bigr) (p)
\end{equation*}
to be the \emph{(periodic) Lyapunov exponent} of $p$. This value does not depend on the choice of a point in the orbit of $p$. Note that $\exp(n \cdot \chi_n(p))$ is the \emph{multiplier} of the periodic orbit $\bigl\{p, \, f(p), \, \ldots, \, f^{n-1}(p) \bigr\}$, and $n \cdot \chi_n(p)$ is the \emph{length} of that orbit \cite{McM10}. 

\begin{comment}

\begin{figure}[ht]
\captionsetup{font=footnotesize}
  \centering
  % Subfigure (A) - Left side
  \begin{subfigure}[b]{0.49\textwidth}
    \includegraphics[width=\textwidth, trim=50 60 50 80, clip]{SinCos30.png}
    \caption{}
    \label{fig:A}
  \end{subfigure}
  \hfill % Adds horizontal space between the two subfigures
  % Subfigure (B) - Right side
  \begin{subfigure}[b]{0.49\textwidth}
    \includegraphics[width=\textwidth, trim=50 60 50 80, clip]{Blaschke30.png}
    \caption{}
    \label{fig:B}
  \end{subfigure}
  \caption{The empirical distributions of the Lyapunov exponents for periodic orbits of period $30$, using $100$ equal-width bins, where the map is:
  \newline
  \hspace*{1em}(A) $x \mapsto 2x+0.01 \big(\sin(2\pi x) + \cos(2\pi x) - 1\big) \mod 1$; 
  \newline
  \hspace*{1em}(B) the Blaschke product $z \mapsto z (z-0.1)/(1-0.1 z)$ restricted to the unit circle; in this case, the corresponding expanding circle map preserves the Lebesgue measure on the circle. }
  \label{Fig:Exp}
\end{figure}
\end{comment}

%\begin{comment}
\begin{figure}[ht]
\captionsetup{font=footnotesize}
  \centering
  % Subfigure (A)
  \begin{subfigure}[b]{0.7\textwidth}
    \includegraphics[width=\textwidth, trim=50 60 50 80, clip]{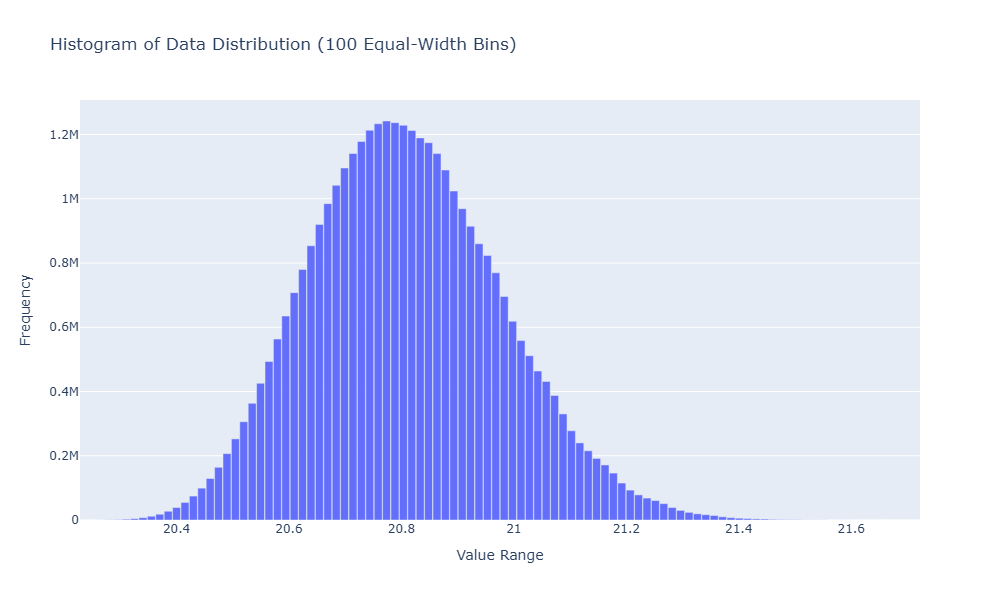}
    \caption{}
    \label{fig:A}
  \end{subfigure}
  
 % \vspace{0.5em} % space between subfigures
  
  % Subfigure (B)
  \begin{subfigure}[b]{0.7\textwidth}
    \includegraphics[width=\textwidth, trim=50 60 50 80, clip]{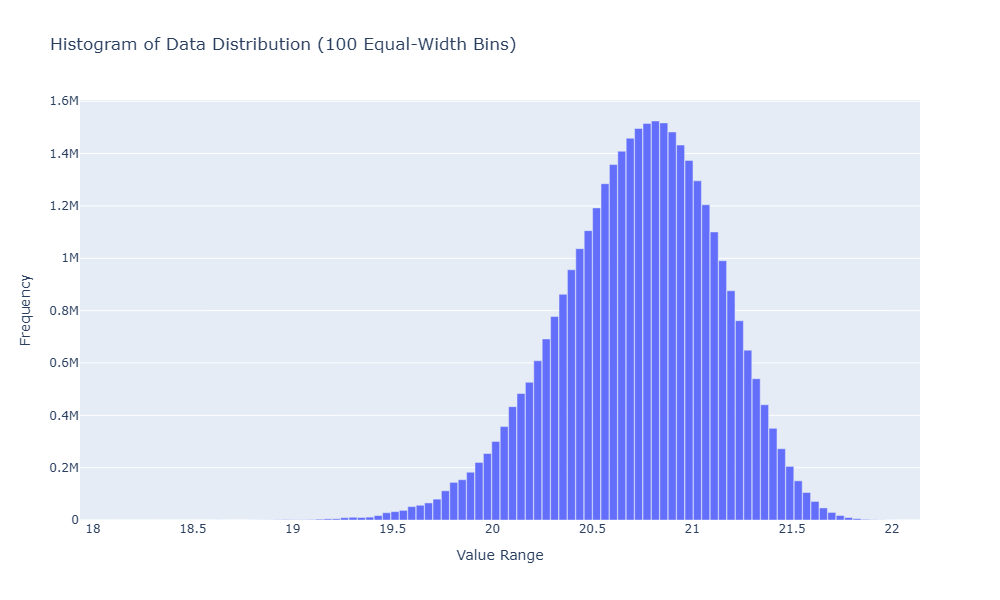}
    \caption{}
    \label{fig:B}
  \end{subfigure}

  \caption{The empirical distributions of the Lyapunov exponents for periodic orbits of period $30$, using $100$ equal-width bins, where the map is:
  \newline
  \hspace*{1em}(A) $x \mapsto 2x+0.01 \big(\sin(2\pi x) + \cos(2\pi x) - 1\big) \mod 1$; 
  \newline
  \hspace*{1em}(B) the Blaschke product $z \mapsto z (z-0.1)/(1-0.1 z)$ restricted to the unit circle; in this case, the corresponding expanding circle map preserves the Lebesgue measure on the circle. }
  \label{Fig:Exp}
\end{figure}
%\end{comment}

For a fixed period $n$, we treat $\chi_n$ as a random variable in $\R$ distributed with the probability
\begin{equation}
\label{Eq:Prob}
\Pillow(\chi_n \in [x,y]) 
\= \frac{\sharp \{p \in \Fix(f^n) : \chi_n(p) \in [x,y]\}}{\sharp \Fix(f^n)},
\end{equation}
where $\sharp A$ stands for the number of elements in a set $A$. 

Let us also denote by $\mu$ the unique invariant probability measure of \emph{maximal entropy} on $\T$ for $f$.

The main result of this article is the following theorem:

\begin{theoremM}[Distribution of periodic Lyapunov exponents]
    Let $f \in \cC^r(\T,\T)$, $r \ge 2$, be a $\cC^r$-smooth orientation-preserving expanding circle map that is not smoothly conjugate to the linear map. Denote by $\overline{\chi}\=\int_{\T} \! \ln f' \, \mathrm{d}\mu$ the Lyapunov exponent of $f$ with respect to the measure of maximal entropy. Then there exist positive constants $\sigma>0$ and $ C > 0$ such that for all $n \in \N$ and $a, b \in \R \cup \{+\infty\}$ with $a<b$, we have
    \begin{equation*}
    \Absbigg{\Pillow\biggl(\frac{(\chi_n - \overline{\chi})\sqrt{n}}{\sigma} \in [a ,b]\biggr) - \frac{1}{\sqrt{2\pi}} \int_a^b \! \exp\biggl(-\frac{x^2}{2}\biggr) \, \mathrm{d}x} \le \frac{C}{\sqrt[4]{n}}.
    \end{equation*}
\end{theoremM}

A closely related result about lengths of closed geodesics on hyperbolic surfaces was obtained by Gekhtman, Taylor, and Tiozzo in \cite{GTT19}, giving a positive answer to the conjecture of Chas--Li--Maskit \cite{CLM13}. More specifically, in \cite{GTT19} the authors establish a version of the Central Limit Theorem (without the error term), as $n \to +\infty$, for the empirical distribution of the lengths of closed geodesics whose combinatorial length (in terms of word length in the generators of the fundamental group) is bounded above by $n$.

Taken together, our main result and the result of \cite{GTT19} provide a new instance of the parallel between the properties of the lengths (Lyapunov exponents) of periodic orbits for expanding circle maps and the behavior of the lengths of closed geodesics on compact hyperbolic surfaces of genus $g \ge 2$. This analogy (``dictionary'') was put forward by McMullen \cite{McM10} in the setting of expanding circle maps arising from Blaschke products.

\section{Preliminaries}

Let $f\in\cC^r(\T, \T)$, $r \ge 2$, $\T = \R/\Z$, be an expanding circle map not smoothly conjugate to a linear map, and $\mathbb P_n^*$ be the partition of $\T$ into half-open intervals on which $f^n$ is invertible. We choose this partition in such a way that $f^n(I) = \T$ for every partition interval $I$. In this way, $\mathbb P_n^*$ is a set of $K^n$ intervals and for every $z \in \T$ each interval $\bbp \in \bP_n^*$ contains exactly one point from the set $f^{-n}(z)$. Furthermore, all but one interval $\bbp \in \bP_n^*$ has a unique point $x_{\bbp} \in \Fix(f^n)$ that corresponds to a periodic orbit of period $n$. Let us denote the exceptional interval $\bbp^*$ and denote the set of non-exceptional intervals by $\bP_n = \bP^*_n \setminus \{\bbp^*\}$. Note that $\sharp \bP_n = K^n-1= \sharp\Fix(f^n)$. 

For each $\bbp \in \bP_n$, we denote by $\vf_\bbp$ the well-defined inverse branch of $f^n|_{\bbp}$. And we denote by $\vf_{\bbp^*}$ the corresponding inverse branch for the exceptional interval $\bbp^*$.

We will also fix the notation 
\begin{equation*}
h(x)\=\ln f'(x) \quad \quad\text{ and }\quad \quad
       g_n(x) \=\frac 1n\sum_{k=0}^{n-1} \bigl( h\circ f^k \bigr) (x), \quad x\in \T.
\end{equation*}
Then  $\chi_n(p)=g_n(p)$ for each $p\in\Fix(f^n)$. We want to study the empirical distribution of ${\chi}_n$ treated as a random variable given by \eqref{Eq:Prob}. For such a random variable, and for each continuous function $\psi\in\cC^0(\R,\R)$, we have the expectation
\begin{equation}
    \label{eq:object}
    \bE(\psi(\chi_n)) =\frac{\sum_{p \in \Fix(f^n)} \psi(\chi_n(p))}{\sharp \Fix(f^n)} = \frac{\sum_{\bbp\in \Pillow_n}\psi(g_n(x_\bbp))}{\sharp \Pillow_n}.
\end{equation}

\medskip
To simplify our notation, in what follows, we will write $X_n = Y_n +\cO(Z_n)$ if there is a constant $C > 0$, that may depend only on the function $f$ (and hence independent of $n$), such that $\abs{X_n - Y_n} \le C \cdot Z_n$. 

Finally, we will denote $\|\cdot\|_{\cC^k}$, $k=0,\,1,\,\dots$, the $\cC^k$-norm on $\T$.

\section{The average}
Our goal is to compute \eqref{eq:object} for $\psi = \id$. To this end, we want to compare it with the transfer operator
\begin{equation}
    \label{eq:transfer_operator}
    \cL \colon \cC^1(\T, \R) \to \cC^1(\T, \R), \quad \quad
    \cL (\vf)(x)\=\sum_{y\in f^{-1}(x)}\varphi(y).
\end{equation}

\begin{lemma}
    \label{lem:E(chi_n)}
    For every $z\in \T$ and every integer $n \geq1$,
    \begin{equation*}
        \bE(\chi_n) = (K^n-1)^{-1} \cdot \cL^n (h)(z) + \cO(K^{-n}).
    \end{equation*}
\end{lemma}
\begin{proof}
    We have 
    \begin{equation*}
        \label{eq:sum_of_g_n_equal_to_sum_of_h}
        \sum_{\bbp\in\Pillow_n}g_n(x_\bbp)=
        \sum_{\bbp\in \bP_n}\frac 1n\sum_{k=0}^{n-1}h \bigl(f^k(x_\bbp) \bigr) = 
        \sum_{\bbp\in\Pillow_n}h(x_\bbp).
    \end{equation*}
    Thus, 
    \begin{align*}
        &\AbsBig{\cL^n (h)(z)-\sum_{\bbp\in \Pillow_n}g_n(x_\bbp)}
        =\AbsBig{\sum_{y\in f^{-n}(z)} h(y)-\sum_{\bbp\in \Pillow_n}h(x_\bbp)}\notag \\
        &\qquad\leq\abs{h(\varphi_{\bbp^*}(z))}+\sum_{\bbp\in\Pillow_n}\abs{(h\circ\varphi_{\bbp})(z)-h(x_\bbp)} \leq \|h\|_{\CC^0} + \|h'\|_{\CC^0}\sum_{\bbp\in\Pillow_n}\abs{\varphi_{\bbp}(z)-x_\bbp}\notag \\
        &\qquad\leq \norm{h}_{\CC^0} + \norm{h'}_{\CC^0} \cdot \sum_{\bbp\in\Pillow_n} \abs{\bbp} \leq \norm{h}_{\CC^0} + \norm{h'}_{\CC^0}.
    \end{align*}
    Since $\bE(\chi_n) = (K^n-1)^{-1} \sum_{\bbp \in \bP_n} g_n(x_\bbp)$, we conclude that $\bE(\chi_n) = (K^n-1)^{-1}\cL^n (h)(z) + \cO(K^{-n})$.
\end{proof}

We introduce the spectral decomposition for the transfer operator \eqref{eq:transfer_operator} in the following lemma (see e.g.\, \cite{DKL21}, \cite[Lemma 3.1]{BKL22}).
\begin{lemma}
    \label{lem:spectral_decomposition}
    For each function $\vf \in \CC^1(\T,\R)$,
    \begin{equation*}
        \cL (\varphi)= K \cdot \mu(\varphi) + Q(\varphi), 
    \end{equation*}
    where $\mu(\vf) \= \int_\T \!\vf(x) \,\mathrm{d}\mu(x)$, with $\mu$ being the measure of maximal entropy. Furthermore, $Q(\mathbbm{1})=0$, $\mu(Q(\vf))=0$, and there exists a constant $C>0$ such that for every integer $n\geq1$,
    \begin{equation*}
        \|Q^n(\varphi)\|_{\CC^1}\leq C\|\varphi\|_{\CC^1}.
    \end{equation*} %\qed
\end{lemma}

Consequently, for each integer $n\geq1$, $\cL^n (\vf)=K^n \cdot \mu(\vf) + Q^n(\varphi)$. Thus the iterations of the normalized operator $K^{-n} \cdot \cL^n(\vf)$ converge to $\mu(\vf)$ exponentially fast (as $n \to +\infty$). Hence Lemma~\ref{lem:E(chi_n)} yields the following corollary. Before stating it, let us introduce the following centered quantities:
\begin{equation}
\label{Eq:Center}
\overline{\chi} \= \mu(h) = \int_{\T}\! \ln f'(x) \,\mathrm{d} \mu(x), \quad \quad \hat h \= h - \mu(h), \quad \quad \hat \chi_n \= \chi_n - \overline{\chi}.
\end{equation}

\begin{cor}[Centered expectation estimate] 
\label{Cor:1}
For every integer $n \ge 1$,
\begin{equation}
\label{eq:expression_of_expectation}
\bE\big(\chi_n\big)=\mu(h)+\cO(K^{-n}), \quad \quad \bE\big(\hat \chi_n\big)=\cO(K^{-n}). 
\end{equation}
In particular, 
\begin{equation*}
\lim_{n \to +\infty} \bE(\chi_n) = \mu(h), \quad \quad \lim_{n \to +\infty} \bE(\hat \chi_n) = 0.
\end{equation*}%\qed
\end{cor}

For the rest of the paper, we focus on the centered random variable $\hat \chi_n$. 

%%%%%%%%%%%%%%%%%%%%%%%%%%%%%%%%%%%%%%%%%%%%%%%%%%%%%%%%%%%%%%%%%%%%%%%%%%%%%%%%%%%%%%%%%%%%%
%%%%%%%%%%%%%%%%%%%%%%%%%%%%%%%%%%%%%%%%%%%%%%%%%%%%%%%%%%%%%%%%%%%%%%%%%%%%%%%%%%%%%%%%%%%%%
\section{The variance}
%%%%%%%%%%%%%%%%%%%%%%%%%%%%%%%%%%%%%%%%%%%%%%%%%%%%%%%%%%%%%%%%%%%%%%%%%%%%%%%%%%%%%%%%%%%%%
%%%%%%%%%%%%%%%%%%%%%%%%%%%%%%%%%%%%%%%%%%%%%%%%%%%%%%%%%%%%%%%%%%%%%%%%%%%%%%%%%%%%%%%%%%%%%

In this section, we want to estimate the variance $\Var(\hat \chi_n) \= \bE\big([\hat \chi_n - \bE(\hat \chi_n)]^2\big) = \bE\big([\chi_n - \bE(\chi_n)]^2\big)$. The key result of this section is the following lemma:

\begin{lemma}[Centered variance estimate]
\label{Lem:Variance}
There exists $\sigma > 0$ such that for every integer $n \ge 1$,
\begin{equation}
    \label{eq:upper_bound_of_the_variance}
    \Var(\hat \chi_n) = n^{-1} \cdot \sigma^2 + \cO\bigl(n^{-2}\bigr).
\end{equation}        
In particular, 
\begin{equation*}
\Var(\sqrt{n} \cdot \hat \chi_n) = \sigma^2+\cO \bigl(n^{-1} \bigr), \quad \quad \lim_{n \to +\infty} \Var \bigl( \sqrt{n} \cdot \hat \chi_n \bigr) = \sigma^2.
\end{equation*}
\end{lemma}

\begin{proof}

Using \eqref{eq:object}, \eqref{Eq:Center}, and Corollary~\ref{Cor:1}, we have:
\begin{equation*}
\Var(\hat \chi_n)=(K^n-1)^{-1} n^{-2}\sum_{\bbp\in \bP_n}\Biggl[\sum_{k=0}^{n-1}\hat h\bigl(f^k(x_\bbp)\bigr)\Biggr]^2+\cO \bigl(K^{-2n} \bigr).
\end{equation*}
But 
\begin{equation*}
\begin{split}
\sum_{\bbp\in \bP_n}\Biggl[\sum_{k=0}^{n-1}\hat h\bigl(f^k(x_\bbp)\bigr)\Biggr]^2 
&= \sum_{\bbp\in \bP_n}\Biggl[\sum_{k=0}^{n-1}\hat h \bigl(f^k(x_\bbp) \bigr)^2+2\sum_{k=0}^{n-2}\sum_{j=k+1}^{n-1}\hat h\bigl(f^k(x_\bbp) \bigr)\hat h \bigl(f^j(x_\bbp)\bigr)\Biggr]\\
&=\sum_{\bbp\in \bP_n}\Biggl[\sum_{k=0}^{n-1}\hat h \bigl(f^k(x_\bbp) \bigr)^2+2\sum_{k=0}^{n-2}\sum_{j=1}^{n-1-k}\bigl( \hat h \cdot \bigl( \hat h\circ  f^j \bigr) \bigr)\bigl(f^k(x_\bbp)\bigr)\Biggr]\\
&=\sum_{\bbp\in \bP_n}\Biggl[\sum_{k=0}^{n-1}\hat h \bigl(f^k(x_\bbp) \bigr)^2+2\sum_{j=1}^{n-1}\sum_{k=0}^{n-j-1}\bigl( \hat h \cdot \bigl( \hat h\circ  f^j \bigr) \bigr)\bigl(f^k(x_\bbp)\bigr)\Biggr]\\
&=\sum_{\bbp\in \bP_n}\Biggl[n \cdot\hat h(x_\bbp)^2+2\sum_{j=1}^{n-1}(n-j)\bigl( \hat h \cdot \bigl( \hat h\circ  f^j \bigr) \bigr)(x_\bbp)\Biggr].
\end{split}
\end{equation*}
Therefore,
\begin{equation}
\label{Eq:Var1}
    \Var(\hat \chi_n) = (K^n-1)^{-1}n^{-1}\sum_{\bbp\in \bP_n}\Biggl[\hat h(x_\bbp)^2+2\sum_{j=1}^{n-1}\biggl(1-\frac{j}{n}\biggr)\bigl( \hat h \cdot \bigl( \hat h\circ  f^j \bigr) \bigr) (x_\bbp)\Biggr] + \cO\bigl(K^{-2n}\bigr).
\end{equation}
Next, we want to compare the expression in \eqref{Eq:Var1} to the $n$-th power of transfer operator~\eqref{eq:transfer_operator}. Note that for every $z \in \T$,
\begin{equation*}
\begin{split}
&\AbsBig{ \cL^n\bigl( \hat h \cdot \bigl( \hat h\circ  f^j \bigr) \bigr)(z)-\sum_{\bbp\in \bP_n}\bigl( \hat h \cdot \bigl( \hat h\circ  f^j \bigr) \bigr)(x_\bbp)} \\
&\qquad=\Absbig{ \bigl( \hat h \cdot \bigl( \hat h\circ  f^j \bigr) \bigr) (\vf_{\bbp^*}(z)) + \sum_{\bbp\in\bP_n} \bigl(\bigl( \hat h \cdot \bigl( \hat h\circ  f^j \bigr) \bigr) (\vf_{\bbp}(z)) -\bigl( \hat h \cdot \bigl( \hat h\circ  f^j \bigr) \bigr)(x_\bbp)\bigr)}\\
&\qquad\le \|\hat h\|^2_{\cC^0} + \|\hat h'\|_{\cC^0}\|\hat h\|_{\cC^0} + \|\hat h\|_{\cC^0} \sum_{\bbp\in\bP_n}  \Absbig{ \bigl( \hat h\circ f^j\circ\vf_{\bbp} \bigr) (z) - \bigl(\hat h\circ f^j \bigr)(x_\bbp)}\\ 
&\qquad\le \|\hat h\|^2_{\cC^0} + \|\hat h'\|_{\cC^0}\|\hat h\|_{\cC^0} +\|\hat h\|_{\cC^0} \|\hat h'\|_{\cC^0}\sum_{\bbp\in\bP_n} \int_{\bbp}\! \bigl(f^j\bigr)' \, \mathrm{d}x \\
&\qquad\le \|\hat h\|^2_{\cC^0} + \|\hat h\|_{\cC^0}\|\hat h'\|_{\cC^0} +\|\hat h\|_{\cC^0} \|\hat h'\|_{\cC^0} \int_{\T} \! \bigl(f^j \bigr)' \, \mathrm{d}x \\
&\qquad\le \|\hat h\|^2_{\cC^0} + \|\hat h\|_{\cC^0}\|\hat h'\|_{\cC^0} \bigl(1 + K^j \bigr).
\end{split}
\end{equation*}
Similarly, for every $z \in \T$,
\begin{equation*}
\begin{split}
\AbsBig{\cL^n \bigl(\hat h^2 \bigr)(z)-\sum_{\bbp\in \bP_n}\hat h(x_\bbp)^2}
&=\AbsBig{\bigl(\hat h \circ \phi_{\bbp^*}\bigr)(z)^2 + \sum_{\bbp\in\bP_n} \bigl(\bigl(\hat h \circ\vf_{\bbp}\bigr)(z)^2 -\hat h(x_\bbp)^2\bigr)} \\
&\le \|\hat h\|^2_{\cC^0} + \|\hat h\|_{\cC^0}\|\hat h'\|_{\cC^0} + \|\hat h\|_{\cC^0} \|\hat h'\|_{\cC^0} \sum_{\bbp\in\bP_n} \abs{\bbp} \\
&\le\|\hat h\|^2_{\cC^0} + 2\|\hat h\|_{\cC^0}\|\hat h'\|_{\cC^0}.
\end{split}
\end{equation*}
Putting the last two computations together into \eqref{Eq:Var1} (and taking into account the linearity of the transfer operator and that $\sum_{j=1}^{n-1} \bigl(1-\frac{j}{n} \bigr)K^j$ is of order $K^n/n$), we obtain:
\begin{equation}
    \label{Eq:Var2}
    \Var(\hat \chi_n) = (K^n - 1)^{-1} n^{-1}\cL^n\biggl(\hat h^2+2\sum_{j=1}^{n-1}\biggl(1-\frac{j}{n}\biggr)\bigl(\hat h \cdot \bigl( \hat h\circ f^j \bigr)\bigr)\biggr)(z)+\cO\bigl(n^{-2}\bigr).
\end{equation}

Let us use the spectral decomposition from Lemma~\ref{lem:spectral_decomposition} to estimate the right-hand side of \eqref{Eq:Var2}. Observe that $\cL^n\bigl(\hat h \cdot \bigl( \hat h\circ f^j \bigr)\bigr) = \cL^{n-j}\bigl( \hat h\cL^j \hat h\bigr)$. Since $\mu\bigl(\hat h\bigr) = 0$, this implies $\cL^n\bigl(\hat h \cdot \bigl( \hat h\circ f^j \bigr)\bigr) = \cL^{n-j}\bigl( \hat h \cdot Q^j\bigl(\hat h\bigr)\bigr)$. Hence,
\begin{equation*}
\sum_{j=1}^n (K^n - 1)^{-1} \frac{j}{n} \Absbig{\cL^n\bigl( \hat h \cdot \bigl( \hat h\circ  f \bigr) \bigr)(z)} = \cO\bigl(n^{-1}\bigr).
\end{equation*}
Thus, \eqref{Eq:Var2} reduces to 
\begin{equation}
    \label{Eq:Var3}
    \Var(\hat \chi_n) = (K^n - 1)^{-1} n^{-1}\cL^n\biggl(\hat h^2+2\sum_{j=1}^{n-1}\bigl(\hat h \cdot \bigl( \hat h\circ f^j \bigr)\bigr)\biggr)(z)+\cO\bigl(n^{-2}\bigr).
\end{equation}

Finally, define the \emph{asymptotic variance of $f$}
\begin{equation}
\label{Eq:Var4}
\sigma^2 
\= \lim_{n \to +\infty} \mu\biggl(\hat h^2+2\sum_{j=1}^{n-1}\bigl(\hat h \cdot \bigl( \hat h\circ f^j \bigr)\bigr)\biggr)
= \mu\biggl(\hat h^2+2\sum_{j=1}^{+\infty}\bigl(\hat h \cdot \bigl( \hat h\circ f^j \bigr)\bigr)\biggr).
\end{equation}
By, e.g., \cite[Equation~(2.11.3)]{PU10}, this limit exists, and by the discussion in \cite[Chapter~2]{PU10}, $\sigma^2 \neq 0$ if and only if $\hat h$ is not cohomologous to $0$. By the marked Lyapunov spectrum rigidity \cite{SS} and Liv\v{s}ic theory, this is possible if and only if $f$ is not smoothly conjugate to a linear map. Hence, by the assumption in the main theorem, $\sigma>0$.

Furthermore, because of the exponential decay of correlations for $f$ \cite{DKL21}, the speed of convergence in \eqref{Eq:Var4} is of order $1/n$, and thus
\begin{equation}
    \label{Eq:Var5}
    \mu\biggl(\hat h^2+2\sum_{j=1}^{n-1}\bigl(\hat h \cdot \bigl( \hat h\circ f^j \bigr) \bigr)\biggr) = \sigma^2+\cO\bigl(n^{-1} \bigr).
\end{equation}

Finally, again by the spectral decomposition, 
\begin{equation*}
\frac{\cL^n\bigl(\hat h^2+2\sum_{j=1}^{n-1}\bigl(\hat h \cdot \bigl( \hat h\circ f^j \bigr) \bigr)\bigr)(z)}{(K^n - 1) n} 
= \frac{K^n  \sigma^2}{(K^n - 1) n}  + \cO\bigl(n^{-2}\bigr)\\ 
= \frac{ \sigma^2}{n} + \cO\bigl(n^{-2}\bigr).
\end{equation*}
Combining this last estimate with \eqref{Eq:Var3}, we get
$\Var(\hat \chi_n) = n^{-1} \cdot \sigma^2 + \cO\bigl(n^{-2}\bigr)$,
as desired.    
\end{proof}

%%%%%%%%%%%%%%%%%%%%%%%%%%%%%%%%%%%%%%%%%%%%%%%%%%%%%%%%%%%%%%%%%%%%%%%%%%%%%%%%%%%%%%%%
%%%%%%%%%%%%%%%%%%%%%%%%%%%%%%%%%%%%%%%%%%%%%%%%%%%%%%%%%%%%%%%%%%%%%%%%%%%%%%%%%%%%%%%%
\section{The distribution and proof of the Main Theorem}
%%%%%%%%%%%%%%%%%%%%%%%%%%%%%%%%%%%%%%%%%%%%%%%%%%%%%%%%%%%%%%%%%%%%%%%%%%%%%%%%%%%%%%%%
%%%%%%%%%%%%%%%%%%%%%%%%%%%%%%%%%%%%%%%%%%%%%%%%%%%%%%%%%%%%%%%%%%%%%%%%%%%%%%%%%%%%%%%%

Let us introduce the random variable $S_n\=\sqrt{n}\hat\chi_n$, and let $\sigma_n^2 \= \Var(S_n)$ be its variance. By Lemma~\ref{Lem:Variance}, $\sigma_n^2 = \sigma^2 + \cO\bigl(n^{-1}\bigr)$, where $\sigma>0$ is the constant from that lemma (the asymptotic variance of $f$).

Let us define the \emph{distribution function}
\begin{equation}\label{eq:Fn def}
F_n(x)\=\Pillow (\{S_n/\sigma \leq x\}) 
= \bE ( \Id_{ \{ S_n/\sigma \le x \} })
\end{equation}
and the corresponding \emph{characteristic function} $\Psi_n \colon \R \to \C$,
\begin{equation*}
\Psi_n(\lambda) \=\bE(\exp (i\lambda  S_n/ \sigma ))=(K^n - 1)^{-1} \sum_{\bbp\in \bP_n}\exp\biggl( \frac{i\lambda}{\sigma \sqrt n}\sum_{k=0}^{n-1} \bigl( \hat h \circ f^k \bigr) (x_\bbp)\biggr).
\end{equation*}
We have $\Psi_n(0)=1$, $\Psi_n'(0)=\cO(\sqrt{n} K^{-n})$ (Corollary~\ref{Cor:1}), $-\Psi_n''(0) = {\sigma_n^2} \big/{\sigma^2} = 1+\cO \bigl(n^{-1} \bigr)$ (Lemma~\ref{Lem:Variance}), and $\Absbig{\Psi_{n}'''(\lambda)} \leq n^{3/2}\norm{\widehat{h}}_{\cC^0} \big/\sigma^3$ for every $\lambda \in \mathbb{R}$. In particular, 
\begin{equation}
    \label{Eq:Psi}
    \Psi_n(\lambda) = 1 - \frac{\lambda^2}{2} + \cO\bigl(\abs{\lambda} \sqrt{n} K^{-n}  +  \abs{\lambda}^2 n^{-1}+\abs{\lambda}^3 n^{ 3/2 }\bigr).
\end{equation}

Let
\begin{equation*}
\cE(\lambda) \= \Absbig{\Psi_n(\lambda) - \exp\bigl(- \lambda^2 \big/2 \bigr)} 
\end{equation*}
be the difference between the characteristic functions of $S_n / \sigma$ and the standard normal distribution. Then, by the Berry--Esseen estimate \cite[Chapter~XVI.3, Equation~(3.13)]{Fe71}, we have
\begin{equation*}
\Absbigg{F_n(x)-\frac{1}{\sqrt{2\pi}}\int_{-\infty}^x \! \exp\biggl(-\frac{y^2}{2}\biggr)\,\mathrm{d}y} 
\le \frac {2}\pi\int_{0}^{T_n} \! \frac{\cE(\zeta)}{\zeta} \,\mathrm{d} \zeta + \frac{24 }{\pi T_n}.
\end{equation*}
Hence, for $T_n = \eps\sqrt[4]{n}$ (where $\eps>0$ will be chosen later uniformly over all $n$),
\begin{equation}
\label{Eq:Int}
F_n(x) 
= \frac{1}{\sqrt{2\pi}}\int_{-\infty}^x \! \exp\biggl(-\frac{y^2}{2}\biggr)\,\mathrm{d}y  
+ C_1 \cdot \int_{0}^{\varepsilon\sqrt[4]{n}} \! \frac{\cE(\zeta)}{\zeta} \,\mathrm{d}\zeta
+ \cO\biggl(\frac{1}{\sqrt[4]{n}}\biggr)
\end{equation}
for some uniform constant $C_1 > 0$. Therefore, to prove the Main Theorem, we need to estimate $\cE(\zeta)$ and show that the middle integral is of order $1/\sqrt[4]{n}$. 

We decompose:
\begin{equation*}
    \int_{0}^{\eps\sqrt[4]{n}} \!\frac{\cE(\zeta)}{\zeta} \,\mathrm{d}\zeta 
    = \int_{0}^{1/n}\! \frac{\cE(\zeta)}{\zeta} \,\mathrm{d}\zeta
    + \int_{1/n}^{\eps\sqrt[4]{n}} \! \frac{\cE(\zeta)}{\zeta} \,\mathrm{d}\zeta 
    \eqqcolon I_1+I_2.
\end{equation*}

As for $I_1$, using \eqref{Eq:Psi} and the expansion $\exp \bigl(-\lambda^2\big/2 \bigr) = 1 - \lambda^2 \big/2 + \cO\bigl(\abs{\lambda}^3\bigr)$, we obtain (again for some uniform constant $C_2>0$)
\begin{equation}
  \label{Eq:Near0}
I_1 = \int_{0}^{1/n} \! \frac{\cE(\zeta)}{\zeta} \,\mathrm{d}\zeta 
\le C_2 \int_{0}^{1/n} \! \biggl(\sqrt{n} \cdot K^{-n} + \frac{\zeta}{n} +\zeta^2 n^{ 3 / 2}\biggr) \,\mathrm{d}\zeta 
= \cO\bigl( n^{-3 / 2} \bigr).  
\end{equation}

As for $I_2$, we cannot apply such a naive strategy. Instead, we will estimate $\cE(\lambda)$ differently. For this purpose, let us introduce a new \emph{twisted operator} acting on $\vf \in \cC^1(\T,\C)$:
\begin{equation*}
\hat \cL_{\lambda/(\sigma\sqrt{n})}(\vf) (x) 
\= K^{-1}\cL \biggl(\exp\biggl( \frac{i\lambda}{\sigma \sqrt{n}}\hat h\biggr) \cdot\vf \biggr) (x)
= K^{-1} \sum_{y \in f^{-1}(x)} \exp\biggl(i\lambda \frac{\hat h(y)}{\sigma\sqrt{n}}\biggr) \cdot\vf(y).
\end{equation*}
Note that $\hat \cL_{0/(\sigma\sqrt{n})} = K^{-1} \cdot \cL$, and hence we want to view $K^{-1} \cdot \hat \cL_{\lambda/(\sigma\sqrt{n})}$ as a \emph{perturbation} of the operator $K^{-1}\cdot \cL$ with perturbation parameter $\lambda / (\sigma\sqrt{n})$. Using the previous notation, for a point $z \in \T$,
\begin{equation*}
\begin{split}
&\Absbig{\Psi_n(\lambda)-\hat\cL_{\lambda/(\sigma\sqrt{n})}^n (\Id)(z)} \\
&\qquad\le \frac{1}{K^{n}} \Absbigg{\sum_{\bbp\in \bP_n}\exp\biggl( \frac{i\lambda}{\sigma \sqrt n}\sum_{k=0}^{n-1} \bigl( \hat h \circ f^k \bigr) (x_\bbp)\biggr)-\exp\biggl( \frac{i\lambda}{\sigma \sqrt n}\sum_{k=0}^{n-1} \bigl( \hat h\circ f^k\circ \vf_\bbp \bigr) (z)\biggr)}  + \frac{2}{K^{n}}\\
&\qquad\le \frac{\abs{\lambda}}{\sigma\sqrt{n} K^n}\sum_{\bbp\in \bP_n} \sum_{k=0}^{n-1}\Absbig{ \bigl( \hat h\circ f^k \bigr) (x_\bbp)-\bigl( \hat h\circ f^k\circ \vf_\bbp\bigr) (z)} + \frac{2}{K^{n}},
\end{split}
\end{equation*}
where the term $2K^{-n}$ comes from a straightforward estimate of the exceptional branch $\vf_{\bbp^*}$ that is present in the transfer operator, but is not accounted for in $\Psi_n$.
%-CK^{-n}

Note that $\abs{f^n(x_\bbp)-(f^n\circ \vf_\bbp)(z)} \le 1$, and thus, by the contraction of the inverse branches of the map, we obtain $\Absbig{f^k(x_\bbp)-\bigl(f^k\circ \vf_\bbp \bigr)(z)}\le \lambda_*^{k-n}$, where $1 < \lambda_* \= \min_{x \in \T} f'(x)$. Hence, there exists a constant $C_3 > 0$, uniform over $n$, such that
\begin{equation*}
\begin{split}
&K^{-n}\sum_{\bbp\in \bP_n} \sum_{k=0}^{n-1} \Absbig{\bigl( \hat h\circ f^k \bigr) (x_\bbp)- \bigl(\hat h\circ f^k\circ \vf_\bbp \bigr)(z)}
\le K^{-n} \|h'\|_{\cC^0} \sum_{\bbp\in \bP_n} \sum_{k=0}^{n-1}\lambda_*^{k-n}
\le C_3.
\end{split}
\end{equation*}
We conclude that for every $z \in \T$,
\begin{equation*}
\Psi_n(\lambda)=\hat\cL_{\lambda/(\sigma\sqrt{n})}^n (\Id) (z)+\cO(\abs{\lambda}/\sqrt{n} ) + \cO(K^{-n}).
\end{equation*}
Therefore,
\begin{equation}
    \label{Eq:E}
    \cE(\zeta) 
    = \Absbig{\Psi_n(\zeta) - e^{-\zeta^2 / 2}} 
    \le \Absbig{\hat\cL_{\zeta/(\sigma\sqrt{n})}^n (\Id) (z) - e^{-\zeta^2 / 2} } + \cO( \zeta/\sqrt{n} ) + \cO(K^{-n})  . 
\end{equation}

To compute $\hat\cL_{\zeta/(\sigma\sqrt{n})}^n (\Id) (z)$, we have to understand its spectrum, in particular the maximal eigenvalue and the derivative of its maximal eigenfunction. This can be done by perturbation theory (see e.g.\ \cite[Section~1]{DKL21}, \cite{Go15}) starting from the spectrum of $K^{-1} \cdot \cL = \hat \cL_0$, which is analyzed in \cite{BKL22}. In this way, we obtain that there exists $\delta >0$ and $\xi \in (0,1)$ so that for $\zeta \in (0,\delta \sqrt{n})$, 
the leading eigenvalue $\kappa(\zeta/(\sigma\sqrt{n}))$ of $\hat \cL_{\zeta/(\sigma \sqrt{n})}$ satisfies
\begin{equation*}
\kappa(t)=1- \sigma^2 t^2/2+\cO\bigl(t^3\bigr), \quad t= \zeta / (\sigma\sqrt{n})
\end{equation*}
(as a perturbation of the leading eigenvalue $1$ for $\hat \cL_0 = K^{-1} \cdot \cL$; here $\sigma$ is exactly the asymptotic variance of $\hat h$). Hence,
\begin{equation*}
\ln \kappa(\zeta / (\sigma \sqrt{n})) 
= - \zeta^2 n^{-1} /2 + \cO\bigl(\zeta^3 n^{-3/2} \bigr), 
\end{equation*}
and we have the following spectral decomposition:
\begin{equation*}
\hat \cL_{\zeta/(\sigma \sqrt{n})}^n(\vf) = (\kappa( \zeta / (\sigma\sqrt{n}) ) )^n \cdot \hat{\Pi}_{\zeta/(\sigma\sqrt{n})}(\vf) + \hat Q_{\zeta/(\sigma\sqrt{n})}^n(\vf),
\end{equation*}
where $\hat{\Pi}_{\zeta/(\sigma\sqrt{n})}$ is the eigenprojection corresponding to $\kappa(\zeta/(\sigma\sqrt{n}))$ and $\hat Q_{\zeta/(\sigma\sqrt{n})} \= \hat \cL_{\zeta/(\sigma \sqrt{n})} - \kappa(\zeta/(\sigma\sqrt{n})) \cdot \hat{\Pi}_{\zeta/(\sigma\sqrt{n})}$. In addition, we have the following estimates:
\begin{gather*}
\hat{\Pi}_{\zeta/(\sigma\sqrt{n})}(\vf) =  \mu(\vf) + \bigO( \norm{\vf}_{\cC^0} \zeta/(\sigma\sqrt{n})) , \\
\Normbig{\hat Q_{\zeta/(\sigma\sqrt{n})}^n(\vf)}_{\cC^1}   
\le C_4\xi^n\norm{\vf}_{\cC^1}.
\end{gather*}
Therefore, for every $z \in \T$,
\begin{equation*}
\hat\cL_{\zeta/(\sigma\sqrt{n})}^n (\Id) (z) 
= \exp\bigl(- \zeta^2/2  + \cO\bigl(\zeta^3/\sqrt{n}\bigr)\bigr)(1 + \bigO(\zeta/\sqrt{n})) + \cO(\xi^n).
\end{equation*}
Note that, for $\delta$ small enough (uniformly in $n$), 
\[
\exp\bigl(- \zeta^2/2  + \cO\bigl(\zeta^3/\sqrt{n}\bigr)\bigr) \le e^{- \zeta^2/2}+\bigl(1-e^{- \zeta^2/2}\bigr)\cO(\zeta/\sqrt{n}) \le e^{- \zeta^2/2}+\cO\bigl(\zeta/\sqrt{n}\bigr),
\]
hence, $\exp\bigl(- \zeta^2/2  + \cO\bigl(\zeta^3/\sqrt{n}\bigr)\bigr)(1 + \bigO(\zeta/\sqrt{n})) \le e^{- \zeta^2/2}+ \bigl(e^{- \zeta^2/2}+1\bigr)\cO(\zeta/\sqrt{n})$.
Combining the estimates above with \eqref{Eq:E}, we get
\begin{equation*}
    \label{Eq:E2}
    \cE(\zeta) \le \zeta \exp\bigl(- \zeta^2/2\bigr)\cO ( 1/\sqrt{n} ) + \cO ( \zeta/\sqrt{n} ) + \cO(\xi^n) + \cO(K^{-n}).
\end{equation*}
Therefore, for $\eps < \delta$,
\begin{equation*}
\begin{split}
    I_2  = \int_{1/n}^{\eps\sqrt[4]{n}} \! \frac{\cE(\zeta)}{\zeta} \,\mathrm{d}\zeta  
         \le \frac{C_5}{\sqrt{n}}\int_{1/n}^{\eps\sqrt[4]{n}} \! e^{-\frac{\zeta^2}{2}}\,\mathrm{d}\zeta + \cO\biggl(\frac{1}{\sqrt[4]{n}}\biggr)+ \cO(n (\xi^n + K^{-n})) .
\end{split}
\end{equation*}
Since the integral in the last estimate is bounded above by $\int_\R \! e^{-\zeta^2\big/2} \, \mathrm{d}\zeta < +\infty$ and $n(\xi^n + K^{-n})$ converges to $0$ exponentially fast, we conclude that $I_2 \le  \cO(1/\sqrt[4]{n})$. Plugging this estimate together with \eqref{Eq:Near0} into \eqref{Eq:Int} yields
\begin{equation*}
F_n(x) 
= \Pillow \biggl(\biggl\{\frac{(\chi_n-\overline{\chi})\sqrt{n}}{\sigma}  \leq x \biggr\} \biggr) 
=  \frac{1}{\sqrt{2\pi}}\int_{-\infty}^x \! \exp\biggl(-\frac{y^2}{2}\biggr)\,\mathrm{d}y  + \cO\biggl(\frac{1}{\sqrt[4]{n}}\biggr),
\end{equation*}
which finishes the proof of the Main Theorem.

\end{document}